\documentclass[a4paper, 10pt, twoside, english]{article}

\usepackage[utf8]{inputenc}
\usepackage[T1]{fontenc}
\usepackage[sc]{mathpazo}
\usepackage[english]{babel}
\usepackage{geometry}
\usepackage{amsmath,amsfonts,amssymb,amsthm}
\usepackage[final,colorlinks]{hyperref}
\usepackage{graphicx}
\usepackage{cleveref}
\hypersetup{
     colorlinks=true,
     linkcolor=black,
     filecolor=blue,
     urlcolor=Mahogany,
     }
\usepackage{mathtools}
\usepackage[dvipsnames]{xcolor}
\usepackage{pdflscape}
\usepackage{etoolbox}
\usepackage{fancyhdr}
\usepackage{lastpage}
\usepackage{ifdraft}
\usepackage[autostyle,italian=guillemets]{csquotes}
\usepackage[backend=bibtex, maxnames=50, firstinits=false, sorting=nyt]{biblatex}
\usepackage{colortbl}
\usepackage{booktabs}
\usepackage{hyphenat}
\usepackage{siunitx}
\usepackage[mathlines,pagewise]{lineno}

\graphicspath{{./images/}}

\usepackage{colortbl}
\usepackage{amssymb}
\usepackage[shortlabels]{enumitem}
\setlist[enumerate,1]{label=\textnormal{(\emph{\roman*})}}
\usepackage{cases}

\definecolor{amber}{rgb}{1.0, 0.75, 0.0}
\definecolor{aogreen}{rgb}{0.0, 0.5, 0.0}
\definecolor{darksienna}{rgb}{0.24, 0.08, 0.08}

\allowdisplaybreaks

\usepackage{array,multirow}

\usepackage{tikz}
\usetikzlibrary{matrix,arrows.meta}

\usepackage{changes} 
\setdeletedmarkup{\color{red}\sout{#1}}

\newcommand{\dd}{\mathop{}\!\mathrm{d}}
\DeclareMathOperator*{\e}{e}

\DeclareMathOperator{\dist}{dist}

\DeclareMathOperator{\Range}{Range}

\newcommand{\R}{\mathbb{R}}
\newcommand{\N}{\mathbb{N}}

\newcommand{\C}{\mathbb{C}}

\theoremstyle{plain}
\newtheorem{theorem}{Theorem}[section]
\newtheorem{lemma}[theorem]{Lemma}

\newtheorem{corollary}[theorem]{Corollary}

\theoremstyle{definition}
\newtheorem{remark}[theorem]{Remark}
\newtheorem{assumption}{Assumption}

\numberwithin{equation}{section}
\numberwithin{table}{section}
\numberwithin{figure}{section}

\usepackage[shortlabels]{enumitem}
\setlist[enumerate,1]{label=\textnormal{(\emph{\roman*})}}

\fancypagestyle{mypagestyle}{%
  \fancyhf{}%
  \fancyhead[LO,RE]{\scriptsize Approximation of reproduction numbers}%
  \fancyhead[RO,LE]{\scriptsize S. De Reggi, F. Scarabel, R. Vermiglio}%
  \fancyfoot[LE,RO]{\scriptsize \thepage\,/\,\pageref*{LastPage}}%
}
\usepackage{authblk}
\usepackage{titling}

\usepackage{footnote}
\thanksmarkseries{arabic}

\title{On the convergence of the pseudospectral approximation of reproduction numbers for age-structured models}

\author{Simone De Reggi\thanks{CDLAb - Computational Dynamics Laboratory, University of Udine, Italy}\thanksgap{2mm}$^,$\thanks{Department of Mathematics, Computer Science and Physics, University of Udine, Italy}\thanksgap{2mm}$^,$\thanksmark{4}
,\ Francesca Scarabel\thanksmark{1}\thanksgap{2mm}$^,$\thanks{Department of Applied Mathematics, University of Leeds, United Kingdom}\thanksgap{2mm}$^,$\thanksmark{4},\ Rossana Vermiglio\thanksmark{1}\thanksgap{2mm}$^,$\thanksmark{2}\thanksgap{2mm}$^,$\thanks{e-mail: \url{simone.dereggi@uniud.it},  \url{f.scarabel@leeds.ac.uk}, \url{rossana.vermiglio@uniud.it}}}
\date{September 2, 2023}

\begin{document}
\maketitle
\begin{abstract}
\noindent We rigorously investigate the convergence of a new numerical method, recently proposed by the authors, to approximate the reproduction numbers of a large class of age-structured population models with finite age span.
The method consists in reformulating the problem on a space of absolutely continuous functions via an integral mapping. 
For any chosen splitting into birth and transition processes, we first define an operator that maps a generation to the next one (corresponding to the Next Generation Operator in the case of $R_0$). Then, we approximate the infinite-dimensional operator with a matrix using pseudospectral discretization.
In this paper, we prove that the spectral radius of the resulting matrix converges to the true reproduction number, and the (interpolation of the) corresponding eigenvector converges to the associated eigenfunction, with convergence order that depends on the regularity of the model coefficients. Results are confirmed experimentally and applications to epidemiology are discussed.


\end{abstract}

\section{Introduction}
Age-structured population models are often formulated as integro-partial differential equations with nonlocal boundary conditions. As a result, they generate infinite-dimensional dynamical systems where, typically, the state space $X$ is a space of $L^1$-integrable functions over the age-interval $(0, a^\dagger)$, where $a^\dagger\in(0, \infty)$ is the maximum age.
Such models have applications in many fields in the life sciences, including ecology, epidemiology and cell biology.
In all these contexts, reproduction numbers play a fundamental role, as they are threshold parameters that determine population/disease persistence or extinction, and can be linked to intervention measures for population/disease control \cite{Lewis2019}. Mathematically, they are characterized as the spectral radius of linear, compact, positive operators  \cite{diekmann1990, Inaba2017, thieme2009spectral}, which are obtained by linearizing the model around an equilibrium and then splitting the linearization into birth and transition \cite{barril2018practical, NGM2010, driessche2002}.  The resulting operators are typically infinite-dimensional, which makes the computation of the reproduction number hardly analytically achievable, in general.  

In \cite{BredaDeReggiScarabelVermiglioWu2022, BredaFlorianRipollVermiglio2021, BredaKuniyaRipollVermiglio2020, Theta2019, Kuniya2017}, numerical methods to approximate the reproduction numbers for age-structured models with finite age span, i.e., with $a^\dagger<+\infty$, are proposed. These methods are based on the idea of separately discretizing the birth and transition operators, and then approximating the reproduction numbers through the spectral radius of a matrix. In \cite{Theta2019, Kuniya2017}, the discretization is obtained via a Theta and a Backward Euler method, respectively, while in \cite{BredaDeReggiScarabelVermiglioWu2022, BredaFlorianRipollVermiglio2021, BredaKuniyaRipollVermiglio2020}, the discretization is obtained via Chebyshev pseudospectral collocation. 
However, all these methods rely on the assumption that the birth and transition operators are defined on a subspace of the state space $X$, including the boundary conditions in the domain of the transition operator (which in general is a differential operator), hence necessarily interpreting the corresponding processes as transition. 

To overcome this lack of flexibility, in \cite{de2024approximating} we introduce a general numerical method to approximate the reproduction numbers for a large class of age-structured models with finite age span, which consists in reformulating the problem on a space of absolutely continuous functions via integration of the age-state, within the extended space framework \cite{Inaba2006, Inaba2017, thieme2009spectral}. On the one hand, this approach permits us to see processes described by boundary conditions as perturbations of an operator with trivial domain condition and, on the other hand, it allows us to work with polynomial interpolation, since point evaluation is well defined. For any given splitting into birth and transition, we discretize the resulting operators via Chebyshev pseudospectral collocation.

In this paper, we rigorously investigate the convergence of the method presented in \cite{de2024approximating}.
To do so, we use the well-established spectral approximation theory of \cite{Chatelin1981}. To prove the norm convergence of the operators in $AC$, we take advantage of the injection of $AC$ into $L^1$, which is compact by the Rellich--Kondrach Theorem \cite[pp. 285, Theorem 9.16]{Brezis2011}, and of interpolation error bounds in the $L^1$-norm \cite[Theorem 1]{Prestin1993}. Note that this approach is not common in the literature as, typically, the convergence of spectral and pseudospectral methods is investigated in the supremum norm for continuous functions or in Hilbert spaces \cite{Boyd2001}.

We prove that the convergence order of the approximating reproduction numbers is driven by the interpolation error on the relevant (generalized) eigenfunctions, which in turn depends on the regularity of the model coefficients. We discuss applications to epidemiological models, which can require to work with piecewise constant coefficients in the case of data-informed parameters.    

The paper is organized as follows. In \cref{s3}, with reference to the prototype linear model considered in \cite{de2024approximating}, we illustrate the reformulation in the space of absolutely continuous functions together with the theoretical framework. In \cref{SectionNumerics}, we recall the numerical method of \cite{de2024approximating} and we prove the well-posedness of the discretized operator. The main contribution of the paper is the convergence proof in \cref{convanal}. In \cref{implementation}, we give  details about the implementation of the method, while \cref{results} contains numerical results and applications to epidemiological models. 

MATLAB demos are available at \url{http://cdlab.uniud.it/software}.

\section{Prototype model and theoretical background}\label{s3}
In this section, we make reference to the prototype linear age-structured population model with finite age span considered in \cite{de2024approximating}, which includes, as particular instances, many models of the literature obtained from the linearization of nonlinear models around an equilibrium. 
For $x(t, \cdot)\in L^1([0, a^\dagger], \R^d)$,  $t\ge 0$ and $d$ a positive integer, the model reads
\begin{equation}\label{prototype-model}
\left\{\setlength\arraycolsep{0.1em}\begin{array}{rlll} 
\mathcal D_{}x(t, a)&=&\displaystyle{\int_0^{a^\dagger}}\beta (a, \alpha) x(t,\alpha) \dd \alpha +\delta(a) x(t,a), & \qquad t\ge 0,\quad a\in [0, a^\dagger],\\ [3mm]
 x(t, 0)&=&\displaystyle\int_0^{a^\dagger}b(a) x(t, a)\dd a,  & \qquad t\ge 0,
\end{array} 
\right.
\end{equation}
where 
\begin{equation*}
\mathcal Dx(t, a):=\partial_t x(t, a)+\partial_a  x(t, a).
\end{equation*}
To conveniently split the inflow processes into two parts,  we assume that $\beta=\beta^++\beta^-$ and $b = b^{+} + b^{-}$, where $\beta^+, \beta^-, b^+, b^-, \delta$ satisfy the following requirements \cite[pp. 77]{Inaba2017}.
\begin{assumption}\label{ass1}
\hspace{0.5mm}
\begin{enumerate}
\item \label{ass11}
$\beta^+, \beta^- \in L^\infty([0,a^\dagger]^2,\R^{d\times d})$ are nonnegative, 
\item \label{ass12}
$b^+, b^-\in L^\infty([0,a^\dagger],\R^{d\times d})$ are nonnegative,
\item \label{ass13} $\delta\in L^\infty([0,a^\dagger],\R^{d\times d})$ is essentially nonnegative with non-positive diagonal elements. 
\end{enumerate} 
\end{assumption}
Note that \Cref{ass1} \ref{ass13} ensures that the fundamental solution matrix associated to $\delta$ is non-negative and non-singular \cite[pp. 77]{Inaba2017}.

Since we are interested in the spectral theory, we enlarge our attention to complex-valued functions, and we consider the Banach spaces $X:=L^1([0, a^\dagger], \C^d)$ and $Y:=AC([0, a^\dagger],\C^d)$, where the latter is equipped with the norm
$\|\psi\|_Y:= |\psi(0)|_{\C^d}+\|\psi'\|_X$. 
The Volterra operator $\mathcal V_0\colon X\to Y_0\subset Y$ defined as
\begin{equation*}\label{isom}
\mathcal V_0\phi(a):=\int_0^a \phi(\alpha)\dd \alpha,\qquad a\in[0, a^\dagger],
\end{equation*}
determines an isomorphism between $X$ and the closed subspace $Y_0=\{\psi\in Y\ |\ \psi(0)=0\}\subset Y$.

By defining $y(t,\cdot):= \mathcal V_0x(t,\cdot)$, \eqref{prototype-model} for $x(t,\cdot) \in X$ is equivalent to the following model
\begin{equation}\label{inteq}
\left\{\setlength\arraycolsep{0.1em}\begin{array}{rlll} 
 \mathcal D y(t, a)&=&\displaystyle\int_0^{a^\dagger}b(\alpha) y(t, \dd\alpha)&\\ &&+ \displaystyle\int_0^a\int_0^{a^\dagger}\beta(\zeta, \alpha) y(t, \dd\alpha)\dd\zeta&\\[3mm]
&&+\displaystyle\int_0^a\delta(\alpha) y(t,\dd\alpha), &\qquad t\ge 0,\quad  a\in [0, a^\dagger],\\ [3mm]
y(t, 0)&=&0,&\qquad t\ge 0.\\
\end{array} 
\right.
\end{equation}
Note that, in \eqref{inteq}, we interpret an absolutely continuous function as a measure.
Now,  we define a bounded linear operator $\mathcal{B}\colon Y \to Y$ accounting for birth
\begin{equation*}
\mathcal B\psi(a):=\int_0^a\int_0^{a^\dagger}\beta^+(\zeta, \alpha)\psi'(\alpha)\dd\alpha\dd\zeta+\int_0^{a^\dagger}b^+(\alpha)\psi'(\alpha)\dd\alpha,  \qquad a\in [0, a^\dagger],
\end{equation*}
and an unbounded linear operator  $\mathcal{M} \colon D(\mathcal M)\subset Y_0\to Y$ accounting for transition 
\begin{align*}
\mathcal M\psi(a):=&\ \psi'(a)-\int_0^a \delta(\alpha)\psi'(\alpha)\dd\alpha\\
&-\int_0^a\int_0^{a^\dagger}\beta^-(\zeta, \alpha)\psi'(\alpha)\dd\alpha\dd\zeta-\int_0^{a^\dagger}b^-(\alpha)\psi'(\alpha)\dd\alpha,\qquad a\in [0, a^\dagger],\\[2mm]
D(\mathcal M):=&\ \{\psi \in Y_0\ |\ \psi'\in Y\}.
\end{align*}
Hereafter, we assume that $\mathcal M$ is invertible with bounded inverse. 
Hence, we can define the operator 
\begin{equation}\label{H}
\mathcal H:= \mathcal{B} \mathcal{M}^{-1}\colon Y\to Y,
\end{equation}  
and characterize the \emph{reproduction number} $R$ for the birth process $\mathcal{B}$ and transition process $\mathcal{M}$ as its spectral radius \cite{de2024approximating, diekmann1990}, i.e., 
\begin{equation}\label{Rdef}
R:=\rho(\mathcal H).
\end{equation}  
Here, we generically refer to ``reproduction number'' to account for several different interpretations, including, as special cases, the basic reproduction number $R_0$ and the type reproduction number $T$, as well as more general definitions \cite{HEESTERBEEK20073,Inaba2017, van2017reproduction}. Note that in \eqref{H} the operator $\mathcal B$ actually acts on $D(\mathcal M)$.

To conclude this section, we prove that the operator $\mathcal H$ in \eqref{H} is compact in $Y$. To do so, we consider the operator $\mathcal W\colon X\to Y\subset X$ defined as
\begin{align*}
\mathcal W\phi(a):=&\int_0^{a}\delta(\alpha)\phi(\alpha)\dd\alpha +\int_0^a\int_0^{a^\dagger}\beta^-(\zeta, \alpha)\phi(\alpha)\dd\alpha\dd\zeta+\int_0^{a^\dagger}b^-(\alpha)\phi(\alpha)\dd\alpha, \qquad a\in[0, a^\dagger],
\end{align*}
and the natural immersion of $Y$ into $X$, $\mathcal J\colon Y\to X$, which is compact due to the Rellich--Kondrach embedding Theorem \cite[pp.~285, Theorem 9.16]{Brezis2011}.  
Note that, thanks to this, the restriction to $Y$ of a linear and bounded operator from $X$ to $Y$ is compact. In the following, to simplify the notation, we omit to write $\mathcal J$ when the domain of application is clear from the context. 

Now, let us observe that, for $\psi\in D(\mathcal M)$, we have 
\begin{equation}\label{MVJ}
\mathcal M\psi=\mathcal M\mathcal V_0\psi'=(\mathcal I_Y-\mathcal W)\psi',\qquad \psi'\in Y.
\end{equation}
Hence, it is easy to see that $\mathcal M$ is invertible with bounded inverse if and only if the operator $\mathcal I_Y- \mathcal W\colon Y\to Y$ is invertible with bounded inverse. 
Moreover, we can prove the following result that will be widely used later on in the paper.
\begin{lemma}\label{mandw}
The operator
\begin{gather*}
\mathcal I_{X}- \mathcal W\colon X\to X
\end{gather*}
is invertible with bounded inverse. Moreover, we have
\begin{equation}\label{explicitminv}
\mathcal M^{-1}=\mathcal V_0\left( \mathcal I_{X}-\mathcal  W\right)^{-1} \mathcal J, 
\end{equation}
and the following inequality holds
\begin{equation}\label{boundonM}
\|\mathcal M^{-1}\|_{Y_0\leftarrow Y} \le a^\dagger\cdot \|( \mathcal I_{X}-\mathcal  W)^{-1}\|_{X\leftarrow X}.
\end{equation}
\end{lemma}
\begin{proof}
Observe that $\mathcal W$ is compact in $X$. Thus, we only need to show that there exists no $\phi\in X\setminus\{0\}$ such that $\phi=\mathcal  W\phi$.
Suppose in fact that $\phi\in X$ solves $\phi=\mathcal W\phi$. 
Then, $\phi\in Y$ and the invertibility of $\mathcal I_Y-\mathcal W$ in $Y$ implies that $\phi=0$.

Now, let us note that given $\xi\in Y$, solving  $\mathcal M\psi=\xi$ in $Y$ for $\psi=\mathcal V_0\psi'$ is equivalent to solve 
\begin{equation*}
 \psi'-\mathcal W\psi'=\mathcal J\xi
\end{equation*}
in $X$ for $\psi'\in Y$, hence \eqref{explicitminv} holds.  Finally, the bound \eqref{boundonM} follows from \eqref{explicitminv} by observing that $\|\mathcal V_0\|_{Y_0\leftarrow X}=1$ and that 
\begin{equation}\label{normJ}
\|\mathcal J \xi\|_X \le a^\dagger \|\xi\|_Y,\qquad \xi \in Y.
\end{equation}
\end{proof}
From \Cref{mandw}, we have that the following diagram commutes
\begin{center}
\begin{tikzpicture}
  \matrix (m)
    [
      matrix of math nodes,
      row sep    = 3em,
      column sep = 6.5em
    ]
    {
      Y              & Y_0 \\
      X            &     X        \\
    };
  \path
    (m-1-1) edge [->] node [left] {$\mathcal J$} (m-2-1)
    (m-1-1.east) edge [->] node [above] {$\mathcal M^{-1}$} (m-1-2)
    (m-2-1.east) edge [->]  node [below] {$(\mathcal I_X-\mathcal W)^{-1}$} (m-2-2)
    (m-2-2) edge [->] node [right] {$\mathcal V_0$} (m-1-2);
\end{tikzpicture}
\end{center}
and, from the compactness of $\mathcal J$, we immediately get the following result.
\begin{corollary}\label{hcompact}
The operator $\mathcal H$ in \eqref{H} is compact and its spectrum consists of eigenvalues only. In addition, if $R$ in \eqref{Rdef} is positive, then it is a dominant real eigenvalue (in the sense of largest in magnitude) with associated a real non-decreasing eigenfunction.
\end{corollary}
\begin{proof}
Since $\mathcal B$ is bounded, it is enough to show that $\mathcal M^{-1}$ is compact. This follows by combining \eqref{explicitminv} with the compactness of $\mathcal J$. 
As a result, the spectrum of $\mathcal H$ consists of eigenvalues only \cite[Theorem 6.8]{Brezis2011}. Finally, the last assertion follows by combining the Krein--Rutmann Theorem \cite{krein1948linear} with the results of \cite[section 2]{de2024approximating}.
\end{proof}

\section{The numerical approach}\label{SectionNumerics}
To derive an approximation of $R$, we construct a finite-dimensional approximation $\mathcal H_N$ of $\mathcal H$, and we approximate the eigenvalues of the latter through those of the former. To do this, we separately discretize the operators $\mathcal B$ and $\mathcal M$ in \cref{s3} via pseudospectral collocation \cite{Boyd2001, Trefethen2000}. In the following, we adopt the MATLAB-like notation according to which elements of a column vector are separated by ``;'' while elements of a row vector are separated by ``,''.

\subsection{Discretization of $\mathcal B$ and $\mathcal M$}
Given a positive integer $N$, we consider the space $Y_N\subset Y$ of algebraic polynomials on $[0, a^\dagger]$ of degree at most $N$ and taking values in $\C^d$, together with
\begin{equation*}
Y_{0,N}:=\{\psi_N\in Y_N\ |\ \psi_N(0)=0\}\subset Y_0.
\end{equation*} 
Let $\Theta_N:=\{a_1<\dots <a_N\}$ denote a mesh of points in $(0, a^\dagger)$  and $\Theta_{0,N}:=\{a_0= 0\}\cup \Theta_N$. We define restriction and prolongation operators
respectively as
\begin{equation*}
\mathcal R_N\colon Y\to \C^{dN},\qquad \mathcal R_N\psi:=(\psi(a_1);\dots; \psi(a_N)),
\end{equation*}
and
 \begin{equation*}
\mathcal P_{0,N}\colon \C^{dN}\to Y_{0,N},\qquad 
\mathcal P_{0,N}\Psi:=\sum_{i=1}^N\ell_{0,i}\Psi_i,
\end{equation*}
where $\Psi:=(\Psi_1;\dots; \Psi_N)$, for $\Psi_i\in\C^d$, $i=1,\dots, N$, and $\{\ell_{0,i}\}_{i=0}^N$ is the Lagrange polynomial basis relevant to $\Theta_{0,N}$.
Observe that 
\begin{equation*}
\mathcal R_N\mathcal P_{0,N}=\mathcal I_{\C^{dN}},\qquad
\mathcal P_{0,N}\mathcal R_{N} = \mathcal L_{0,N},
\end{equation*}
where $\mathcal L_{0,N}\colon Y_0\to Y_{0,N}$ is the Lagrange interpolation operator relevant to $\Theta_{0,N}$.

Then, we derive the finite-dimensional approximations $\mathcal B_{N}, \mathcal M_{N}\colon \C^{dN}\to \C^{dN}$ of $\mathcal B$ and $\mathcal M$, respectively, as 
\begin{equation*}
    \mathcal B_{N}:=\mathcal R_{N}\mathcal B \mathcal P_{0,N}, \qquad \mathcal M_{N}:=\mathcal R_{N}\mathcal M \mathcal P_{0,N}.
\end{equation*}
In order to derive a finite-dimensional approximation $\mathcal H_N$ of $\mathcal H$, in the next section we show that there exists a positive integer $\bar N$ such that $\mathcal M_N$ is invertible for every integer $N\ge \bar N$ under the following assumption.
\begin{assumption}\label{assCheb}
$\Theta_N$ is the mesh of Chebyshev zeros \cite{Trefethen2013}.
\end{assumption}
Subsequently, $\mathcal H_N\colon\C^{dN}\to \C^{dN}$ is defined as 
\begin{equation*}
\mathcal H_N:=\mathcal B_{N}\mathcal M_{N}^{-1},\qquad N\ge \bar N.
\end{equation*} 

\subsection{Invertibility of $\mathcal M_N$}
We consider the Lagrange polynomial basis relevant to $\Theta_N$, $\{\ell_i\}_{i=1}^N$, and define the prolongation operator
\begin{equation*}
{\mathcal P}_{N-1}\colon \C^{dN}\to Y_{N-1}\subset Y,\qquad {\mathcal P}_{N-1}\Phi:=\sum_{i=1}^N\ell_i\Phi_i,
\end{equation*}
where $\Phi:=(\Phi_1;\dots; \Phi_N)$. 
Note that 
\begin{equation*}
\mathcal R_N\mathcal P_{N-1}=\mathcal I_{\C^{dN}},\qquad
\mathcal P_{N-1}\mathcal R_{N} = \mathcal L_{N-1},
\end{equation*}
where $\mathcal L_{N-1}\colon Y\to Y_{N-1}$ is the Lagrange interpolation operator relevant to $\Theta_{N}$.

Then, we observe that $\mathcal M_N$ is invertible if and only if, for every $\Xi\in \C^{dN}$, there exists a unique $\Psi\in \C^{dN}$ such that 
\begin{equation}\label{disceq}
\mathcal M_N\Psi=\Xi.
\end{equation}
In this case, let $\psi_N:=\mathcal P_{0,N}\Psi$. From $\psi_N=\mathcal V_0\psi'_N$, \eqref{MVJ} and \eqref{disceq}, we get
\begin{equation*}
\mathcal R_N \mathcal M\mathcal P_{0,N} \Psi=\mathcal R_N(\mathcal I_Y-\mathcal W)\psi'_N=\Xi,
\end{equation*}
and by applying $\mathcal P_{N-1}$, since $\psi'_N=\mathcal L_{N-1}\psi'_N$, we obtain
\begin{equation}\label{2}
(\mathcal I_Y-\mathcal L_{N-1}\mathcal W)\psi'_N=\mathcal P_{N-1}\Xi.
\end{equation}
Now, we prove that $\mathcal M_N$ is invertible by showing that \eqref{2} has a unique polynomial solution $\psi_N'$ in $X$.

\begin{lemma}\label{5}
Under \Cref{assCheb}, we have that  
\begin{equation}\label{convLn}
\left\|(\mathcal I_{X}-{\mathcal L}_{N-1}\mathcal W)
-(\mathcal I_{X}-\mathcal W)\right\|_{X\leftarrow X}\to 0,\qquad N\to \infty.
\end{equation}
Moreover, there exists a positive integer $\bar N$ such that, for every integer $ N\ge \bar N$,  $\mathcal I_{X}-{\mathcal L}_{N-1}\mathcal W$ is invertible and
\begin{equation}\label{boundLn}
\left\|(\mathcal I_{X}-{\mathcal L}_{N-1}\mathcal W)^{-1}\right\|_{X\leftarrow X}
\le 2\left\|(\mathcal I_{X}-\mathcal W)^{-1}\right\|_{X\leftarrow X}.
\end{equation}
Finally, $\mathcal M_N\colon \C^{dN}\to \C^{dN}$ is invertible for every integer $N\ge \bar N$ with inverse
\begin{align}\label{inversematrix}
\mathcal M_N^{-1}:=&\mathcal R_N\mathcal V_0\left(\mathcal I_{X} -{\mathcal L}_{N-1}\mathcal W\right)^{-1} \mathcal P_{N-1}.
\end{align}
\end{lemma}
\begin{proof}
We use a proof technique similar to \cite[Chapter 5]{breda2014stability}. Let us observe that 
\begin{gather*}
\mathcal I_{X}-{\mathcal L}_{N-1}\mathcal W
=(\mathcal I_{X}-\mathcal W)+\left(\mathcal I_{X}-{\mathcal L}_{N-1}\right)\mathcal W.
\end{gather*}
Now, since $\Range(\mathcal W)\subset Y$ and $\mathcal W\colon X\to Y$ is bounded, \Cref{assCheb} 
 guarantees  that there exists a positive constant $C$ independent of $N$ such that \cite[Theorem 1]{Prestin1993} 
\begin{align}\label{primoprestin}
\|\left(\mathcal I_{X}-{\mathcal L}_{N-1}\right)\mathcal W\|_{X\leftarrow X}
\le C\cdot \|\mathcal W\|_{Y\leftarrow X}\cdot \frac{\log N}{N}\to 0,\qquad N\to \infty.
\end{align}
This gives \eqref{convLn}. Then, the Banach perturbation Lemma \cite[Theorem 10.1]{Kress1989} ensures that there exists $\bar N\in \N$ such that $\mathcal I_{X}-{\mathcal L}_{N-1}\mathcal W$ is invertible in $X$ and \eqref{boundLn} holds 
for every integer $N\ge \bar N$. Finally, \eqref{inversematrix} follows from \eqref{disceq} and \eqref{2}.
\end{proof}
Note that, if $\Xi=\mathcal R_N \xi$, for $\xi \in Y$, then the unique polynomial solution $\psi'_N$ of \eqref{2} is given by
$\psi_N'=(\mathcal I_X- \mathcal L_{N-1}\mathcal W)^{-1}\mathcal J\mathcal L_{N-1}\xi$.
Hence, by introducing the operator $\widehat{\mathcal M}_N^{-1}\colon Y\to Y_{0,N}$ defined as
\begin{equation}\label{continvM}
\widehat{\mathcal M}_N^{-1}:= \mathcal V_0(\mathcal I_X- \mathcal L_{N-1}\mathcal W)^{-1}\mathcal J\mathcal L_{N-1},
\end{equation}
we have that $\psi_N=\widehat{\mathcal M}_N^{-1}\xi$.

\section{Convergence analysis}\label{convanal}
In this section, we investigate the convergence of the eigenvalues (and the corresponding eigenvectors) of $\mathcal H_N$ to those of $\mathcal H$ via the spectral approximation theory of \cite{Chatelin1981}. 
To this aim, let us define the operator 
\begin{equation*}
\widehat{\mathcal H}_N:={\mathcal P}_{N-1}\mathcal H_N\mathcal R_N\colon Y\to Y,
\end{equation*}
which has the same nonzero eigenvalues with the same geometric and partial multiplicities of $\mathcal H_N$ \cite[Proposition 4.1]{BredaMasetVermiglio2012}. 
Then, since $\widehat{\mathcal H}_N-\mathcal H$ is compact for every $N\ge\bar N$ thanks to \Cref{hcompact}, it is sufficient to show that \cite[pp. 498]{Chatelin1981}
\begin{equation}\label{convnormH}
\|\widehat{\mathcal H}_N-\mathcal H\|_Y\to 0, \qquad N\to \infty.
\end{equation}
From \eqref{continvM}, we have that $\mathcal L_{0,N}\widehat{\mathcal M}_N^{-1}=\widehat{\mathcal M}_N^{-1}$, hence we can write
\begin{align*}
\widehat{\mathcal H}_N={\mathcal L}_{N-1}\mathcal B\widehat{\mathcal M}_N^{-1},
\end{align*}
and 
\begin{align*}
\widehat{\mathcal H}_N-\mathcal H
=({\mathcal L}_{N-1}-\mathcal I_Y)\mathcal B (\widehat{\mathcal M}_N^{-1}-\mathcal M^{-1})+({\mathcal L}_{N-1}-\mathcal I_Y)\mathcal B\mathcal M^{-1}
+\mathcal B (\widehat{\mathcal M}_N^{-1}-\mathcal M^{-1}).
\end{align*} 
Thus, in order to prove the convergence of $\widehat{\mathcal H}_N $ to $\mathcal H$, we need to investigate the behavior of $({\mathcal L}_{N-1}-\mathcal I_Y)\mathcal B$ and $ \widehat{\mathcal M}_N^{-1}-\mathcal M^{-1}$ as $N\to \infty$.

\subsection{Convergence in norm of $\widehat{\mathcal H}_N $ to $\mathcal H$}
\begin{lemma}\label{A}
Let \Cref{assCheb} hold. Then  
\begin{equation*}
\|\widehat{\mathcal M}_N^{-1}-\mathcal M^{-1}\|_{Y_0\leftarrow Y}\to 0,\qquad N\to \infty.
\end{equation*}
\end{lemma}
\begin{proof}
Observe that, from \eqref{explicitminv} and \eqref{continvM}, we can write $\widehat{\mathcal M}_N^{-1}-\mathcal M^{-1}$ as
\begin{align}
\label{M1}\widehat{\mathcal M}_N^{-1}-\mathcal M^{-1}=&\ \mathcal V_0[(\mathcal I_{X} -{\mathcal L}_{N-1}\mathcal W)^{-1}\mathcal J({\mathcal L}_{N-1}-\mathcal I_{Y})\\
\label{M2}&+\left((\mathcal I_{X} -{\mathcal L}_{N-1}\mathcal W)^{-1}-(\mathcal I_{X} -\mathcal W)^{-1}\right)\mathcal J].
\end{align}
Simple computations show that
\begin{equation}
\label{invespezz}(\mathcal I_{X} -{\mathcal L}_{N-1}\mathcal W)^{-1}
-(\mathcal I_{X}-\mathcal W)^{-1}=(\mathcal I_{X} -{\mathcal L}_{N-1}\mathcal W)^{-1} ({\mathcal L}_{N-1}-\mathcal  I_X)\mathcal W(\mathcal I_{X}-\mathcal W)^{-1}.
\end{equation}
Hence, from \eqref{normJ} and \eqref{boundLn}, we get
\begin{align}
\label{boundMN1} \|\widehat{\mathcal M}_N^{-1}-\mathcal M^{-1}\|_{Y_0\leftarrow Y}
\le\ &C_1\left\|
\mathcal J({\mathcal L}_{N-1}-\mathcal I_{Y})\right\|_{X\leftarrow Y}\\ \label{boundMN2}
&+\frac{a^\dagger C_1^2}{2} \|({\mathcal L}_{N-1}-\mathcal  I_X)\mathcal W\|_{X\leftarrow X},
\end{align}
where $C_1:=2\|(\mathcal I_{X}-\mathcal W)^{-1}\|_{X\leftarrow X}$.
The term in \eqref{boundMN2} tends to zero as $N\to \infty$ thanks to \eqref{primoprestin}.
As for the term on the right-hand side of \eqref{boundMN1}, \Cref{assCheb} ensures that there exists a positive constant $C_2$ independent of $N$ such that \cite[Theorem 1]{Prestin1993}
\begin{align*}
\|\mathcal J({\mathcal L}_{N-1}-\mathcal I_{Y})\|_{X\leftarrow Y}\le C_2\cdot \frac{\log N}{N}\to 0,\qquad N\to \infty.
\end{align*}
The thesis follows.
\end{proof}
Now, in order to prove \eqref{convnormH}, we make the following assumption.
\begin{assumption}\label{beta}
$\beta^+(\cdot, \alpha)\in C([0, a^\dagger], \R^{d\times d})$ for almost all $\alpha\in [0, a^\dagger]$. 
\end{assumption}
\begin{theorem}
Let \Cref{assCheb} and \Cref{beta} hold. Then $\|\widehat{\mathcal H}_N -\mathcal H\|_Y\to 0$ as $N\to \infty$.
\end{theorem}
\begin{proof}
Let us observe that
\begin{align}
\label{a}\|\widehat{\mathcal H}_N -\mathcal H\|_{Y\leftarrow Y}
&\le\|({\mathcal L}_{N-1}-\mathcal I_Y)\mathcal B\|_{Y\leftarrow Y_0}\cdot \|\widehat{\mathcal M}_N^{-1}-\mathcal M^{-1}\|_{Y_0\leftarrow Y}\\
\label{b}
&+\|({\mathcal L}_{N-1}-\mathcal I_Y)\mathcal B\mathcal M^{-1}\|_{Y\leftarrow Y}\\
\label{c}
&+\|\mathcal B \|_{Y\leftarrow Y_0}\cdot \|\widehat{\mathcal M}_N^{-1}-\mathcal M^{-1}\|_{Y_0\leftarrow Y}.
\end{align} 
The term in \eqref{c} tends to zero as $N\to \infty$ thanks to \Cref{A}. As for \eqref{a} and \eqref{b}, let us recall that for every $\psi\in Y_0$ we have
\begin{align}\label{Bbound}
\|({\mathcal L}_{N-1}-\mathcal I_Y)\mathcal B\psi\|_{Y}&=\left|({\mathcal L}_{N-1}-\mathcal I_Y)\mathcal B\psi_{|_{a=0}}\right|_{\C^{d}}+\left\|\left(({\mathcal L}_{N-1}-\mathcal I_Y)\mathcal B\psi\right)'\right\|_X.
\end{align}
The first term on the right-hand side of \eqref{Bbound} tends to zero as $N\to \infty$ since $\mathcal BY_0\subseteq Y$ ensures the convergence of the Lagrange interpolation at the Chebyshev zeros in the infinite-norm \cite[Theorem 1]{Krylov1956}. As for the second term, \Cref{beta} guarantees that $(\mathcal B \psi)'\in C([0, a^\dagger], \C^d)$, from which it follows that \cite[Theorem 1]{Szabados1990}
\begin{equation}\label{termB}
\left\|({\mathcal L}_{N-1}\mathcal B\psi)'-(\mathcal B\psi)'\right\|_X\to 0,\qquad N\to \infty.
\end{equation}
Since $\mathcal M^{-1}$ is compact, this implies that \eqref{b} tends to zero as $N\to \infty$ \cite[Theorem 10.6]{Kress1989}. Finally, from the uniform boundedness principle \cite[Theorem 2.2]{Brezis2011}, we get
\begin{equation}\label{unifboundB}
\sup_{N\in\N}{\|({\mathcal L}_{N-1}-\mathcal I_Y)\mathcal B\|_{Y\leftarrow Y_0}}<\infty.
\end{equation}
Hence, \eqref{a} tends to zero as $N\to \infty$ thanks to  \Cref{A}. 
\end{proof}

\subsection{Convergence of the eigenvalues and the eigenspaces}
\begin{theorem}\label{teospec}
Let \Cref{assCheb} and \Cref{beta} hold. Let $\lambda\in\C$ be an isolated nonzero eigenvalue of $\mathcal H$ with finite algebraic multiplicity $m$ and ascent $l$ and let $\Delta$ be a neighborhood of $\lambda$ such that $\lambda$ is the sole eigenvalue of $\mathcal H$ in $\Delta$. Then there exists $\bar N$ such that, for $N\ge \bar N$, $\widehat{\mathcal H}_N$ has in $\Delta$ exactly $m$ eigenvalues $\lambda_{N, i}$, $i=1,\dots, m$, counting their multiplicities. Moreover, 
\begin{align*}
\max_{i=1, \dots, m}|\lambda_{N, i}-\lambda|=O\big(\varepsilon_N^{1/l}\big)
\end{align*}
where
\begin{equation}\label{epsilon}
\varepsilon_N:=\|\mathcal {\widehat H}_N-\mathcal H\|_{Y\leftarrow \mathcal M_{\lambda}}
\end{equation}
and $\mathcal M_{\lambda}$ is the generalized eigenspace of $\lambda$. 
Finally, for any $i=1,\dots, m$ and for any eigenfunction $\psi_{N, i}$ of $\mathcal {\widehat H}_N$ relevant to $\lambda_{N,i}$ such that $\|\psi_{N,i}\|_Y=1$, we have
\begin{equation*}
\dist(\psi_{N,i},\  \ker(\lambda \mathcal I_Y-\mathcal H))=O(\varepsilon_N^{1/l}),
\end{equation*}
where $\dist$ is the distance in the space $Y$ between an element and a subspace.
\end{theorem}
\begin{proof}
The thesis follows from \cite[Proposition 2.3 and Proposition 4.1]{Chatelin1981}.
\end{proof}
Now, we complete the analysis with error bounds under each of the following regularity conditions.

\begin{assumption}\label{ass2}
\hspace{2mm}
\begin{enumerate}
\item \label{ass21} $\beta^+, \beta^- \in W^{s, \infty}([0,a^\dagger]^2,\R^{d\times d})$ and $\delta\in W^{s, \infty}([0,a^\dagger],\R^{d\times d})$  for some integer $s\ge 1$,
\item \label{ass22} $\beta^+, \beta^- \in C^\infty([0,a^\dagger]^2,\R^{d\times d})$ and $\delta\in C^\infty([0,a^\dagger],\R^{d\times d})$,
\item \label{ass23} $\beta^+, \beta^-, \delta$ are real analytic.
\end{enumerate} 
\end{assumption}

\begin{corollary}\label{corollbound}
Let $\lambda$, $\lambda_{N,i}$, $i=1,\dots, m$, $\varepsilon_N$, and $\mathcal M_\lambda$  be as in \Cref{teospec}, and let \Cref{assCheb} and \Cref{ass2} hold. 
Then $\varepsilon_N=O(\rho_N)$, where
\begin{equation*}
\rho_N:=\begin{cases}
N^{-s}\log N\quad&\text{under \Cref{ass2} \ref{ass21}},\\
N^{-r}\log N\quad&\text{for every integer $r\ge 1$ under \Cref{ass2} \ref{ass22}},\\
p^{-N}\log N\quad&\text{for some constant $p>1$ under \Cref{ass2} \ref{ass23}}.
\end{cases}
\end{equation*}
\end{corollary}
\begin{proof}
For $\psi_\lambda \in \mathcal M_\lambda$,  from \eqref{boundLn}, \eqref{M1}-\eqref{invespezz}, \eqref{a}-\eqref{c} and \eqref{unifboundB}, we have 
\begin{align}
\label{boundH1}\|(\widehat{\mathcal H}_N-\mathcal H)\psi_\lambda\|_{Y}\le &\ C\big(\|\mathcal J(\mathcal L_{N-1}-\mathcal I_Y)\psi_\lambda\|_{X}\\
\label{boundH2}&+\| ({\mathcal L}_{N-1}-\mathcal  I_X)\mathcal W(\mathcal I_{X}-\mathcal W)^{-1}\psi_\lambda\|_{X}\big)\\
\label{boundH3}& +\|({\mathcal L}_{N-1}-\mathcal I_Y)\mathcal H\psi_\lambda\|_{Y},
\end{align}
where $$C:=2\|(\mathcal I_{X} -\mathcal W)^{-1}\|_{X\leftarrow X}\left(\sup_{N\in\N}{\|({\mathcal L}_{N-1}-\mathcal I_Y)\mathcal B\|_{Y\leftarrow Y_0}}+\|\mathcal B\|_{Y\leftarrow Y_0}\right).$$ 
Moreover, for $\mathcal T\in \{\mathcal I_Y, \mathcal H, \mathcal W(\mathcal I_X-\mathcal W)^{-1}\mathcal J\}$, we have 
\begin{enumerate}
\item $\mathcal T(\mathcal M_{\lambda}) \subset W^{s+1, \infty}([0, a^\dagger], \C^{d\times d})$ under \Cref{ass2} \ref{ass21}, 
\item $\mathcal T(\mathcal M_{\lambda})\subset C^{\infty}([0, a^\dagger], \C^{d\times d})$ under \Cref{ass2} \ref{ass22},
\item $\mathcal T(\mathcal M_{\lambda})$ consists of analytic functions under \Cref{ass2} \ref{ass23}.
\end{enumerate}
Hence, we can bound the term on the right-hand side of \eqref{boundH1} as follows
\begin{equation*}
\|\mathcal J(\mathcal L_{N-1}-\mathcal I_Y)\psi_\lambda\|_X \le a^\dagger \|(\mathcal L_{N-1}-\mathcal I_Y)\psi_\lambda\|_\infty,
\end{equation*}
and, from Jackson's type theorems \cite[section 1.1.2]{Rivlin1981} we get  
\begin{equation*}
\|(\mathcal L_{N-1}-\mathcal I_Y)\psi_\lambda\|_\infty
\le O\left( (1+\Lambda_{N-1}) E_{N-1}(\psi_\lambda)\right)
=O(\rho_N),
\end{equation*}
where $\Lambda_{N-1}$ is the Lebesgue constant relevant to $\Theta_N$, that under \Cref{assCheb} is $O\left(\log(N)\right)$, and $E_{N-1}(\psi_\lambda)$ is the best uniform approximation error of $\psi_\lambda$ in the space of polynomials of degree at most $N-1$. 
The term in \eqref{boundH2} can be bounded as 
\begin{equation*}
\|(\mathcal L_{N-1}-\mathcal I_X)\mathcal W (\mathcal I_X-\mathcal W)^{-1}\psi_\lambda\|_X \le\ a^\dagger \|(\mathcal L_{N-1}-\mathcal I_Y)\mathcal W(\mathcal I_X-\mathcal W)^{-1}\psi_\lambda\|_\infty,
\end{equation*}
which, in turn, can be bounded as
\begin{equation*}
\|(\mathcal L_{N-1}-\mathcal I_Y)\mathcal W(\mathcal I_X-\mathcal W)^{-1}\psi_\lambda\|_\infty
\le O\left( (1+\Lambda_{N-1}) E_{N-1}(\mathcal W(\mathcal I_X-\mathcal W)^{-1}\psi_\lambda)\right)
=O(\rho_N).
\end{equation*}
Finally, the term in \eqref{boundH3} can be bounded as
\begin{align*}
\|({\mathcal L}_{N-1}-\mathcal I_Y)\mathcal H\psi_\lambda\|_{Y}
\le&\  
\left\|({\mathcal L}_{N-1}-\mathcal I_Y)\mathcal H\psi_\lambda\right\|_{\infty}+a^\dagger\left\|\left(({\mathcal L}_{N-1}-\mathcal I_Y)\mathcal H\psi_\lambda\right)'\right\|_\infty.
\end{align*}
For these, from \cite[section 1.1.2]{Rivlin1981} we get
\begin{gather*}
\left\|({\mathcal L}_{N-1}-\mathcal I_Y)\mathcal H\psi_\lambda\right\|_{\infty}=O\left((1+\Lambda_{N-1})E_{N-1}(\mathcal H\psi_\lambda)\right)=O(\rho_N),
\end{gather*}
and, from \cite[Theorem 4.2.11]{Mastroianni} we get
\begin{gather*}
\left\|(({\mathcal L}_{N-1}-\mathcal I_Y)\mathcal H\psi_\lambda)'\right\|_{\infty}=
O\left((1+\Lambda_{N-1})E_{N-2}\left((\mathcal H\psi_\lambda)'\right)\right)=O(\rho_N).
\end{gather*}
The thesis follows from the Riesz theory for compact operators \cite[section 3.1]{Kress1989}, which ensures that  the generalized eigenspaces of $\mathcal H$ have finite dimension, see \cite[Proposition 4.9]{liessi2018}.
\end{proof}

\begin{remark}\label{remOrder}
One can show that, under \Cref{ass2} \ref{ass21}, the convergence of $R_N = \rho(\mathcal{H}_N)$ to $R$ gains one order compared to $\rho_N$, i.e., it has order $O((N^{-(s+1)} \log N)^{1/l})$.
In fact, using the birth operator $\mathcal B$ and the transition operator $ \mathcal M$, we can consider the compact operator $\mathcal M^{-1}\mathcal B: Y \to D(\mathcal{M})$ whose eigenvalues $\lambda$ coincide with those of $\mathcal{H}.$ The eigenfunctions $\varphi _\lambda$ of $\mathcal M^{-1}\mathcal B$ (corresponding to $\lambda$) belong to $D(\mathcal{M})$, and satisfy $\varphi_\lambda = \mathcal M^{-1} \psi_\lambda,$ where $\psi_\lambda$ are the corresponding eigenfunctions of $\mathcal{H}.$ The reproduction number $R$ is also the spectral radius of $\mathcal M^{-1}\mathcal B,$ providing an alternative way to approximate it through the spectral radius of the matrix $\mathcal M_N^{-1}\mathcal B_N$. It is easy to see that the matrices $\mathcal M_N^{-1}\mathcal B_N$ and $\mathcal B_N\mathcal M_N^{-1}$ are similar, hence they have the same eigenvalues for every $N \geq 1$. The convergence analysis can be carried out as before. The operator $\mathcal P_{N,0}\mathcal M_N^{-1}\mathcal B_N \mathcal R_N:  Y \to  Y_{0, N}$ has the same nonzero eigenvalues with the same geometric and partial multiplicities of $\mathcal M_N^{-1}\mathcal B_N$ \cite[Proposition 4.2]{BredaMasetVermiglio2012}, and from \eqref{continvM} it can be expressed as $\widehat{\mathcal M}_N^{-1}\mathcal B \mathcal L_{N,0}$. Moreover, since $\Range(\widehat{\mathcal M}_N^{-1}) \subset Y_{0,N}$, $\widehat{\mathcal M}_N^{-1}\mathcal B  \mathcal L_{N,0}$ has  the same eigenvalues with the same geometric and partial multiplicities, and the same eigenfunctions as the operator $\widehat{\mathcal M}_N^{-1}\mathcal B$ \cite[Proposition 4.3]{BredaMasetVermiglio2012}.
The norm convergence of $\widehat{\mathcal M}_N^{-1}\mathcal B$ to $\mathcal M^{-1}\mathcal B$ easily follows from
\Cref{A}, and an analogous of \Cref{teospec} can be derived. The main difference is that, since the convergence order of the approximation error on the eigenvalue $\lambda$ is driven by the interpolation error on $\mathcal M^{-1}\mathcal B\varphi_\lambda$, which under \Cref{ass2} \ref{ass21} 
has one more degree of smoothness compared to $\mathcal H\psi_\lambda$, we get that $|R-R_N|=O((N^{-(s+1)}\log N)^{1/l})$. \nameref{Example 1} in section \ref{results} illustrates this behavior.
\end{remark}

\section{Implementation issues}\label{implementation}
Here we give an explicit description of the entries of the matrices $\mathcal B_N$ and $\mathcal M_N$.
For the sake of simplicity, we restrict to the case $d=1$.

Thanks to the cardinal property of the Lagrange polynomials ($\ell_{0,j}(a_i)=\delta_{ij}$, $i,j=0,\dots, N$, where $\delta_{ij}$ is the Kronecker's Delta), it is easy to see that the entries of the matrices are explicitly given by
\begin{align}\label{B_N_matrix}
(\mathcal B_{N})_{ij}=&\int_0^{a_i}\int_0^{a^\dagger}\beta^+(\zeta, \alpha)\ell_{0,j}'(\alpha)\dd \alpha\dd \zeta&\\\notag&+\int_0^{a^\dagger}b^+(\alpha)\ell_{0,j}'(\alpha)\dd \alpha,\qquad
&i,j=1,\dots N,    
\end{align} 
and
\begin{align}\label{M_N_matrix}
(\mathcal M_{N})_{ij}=\ &\ell_{0,j}'(a_i)-\int_0^{a_i}\delta(\alpha)\ell_{0,j}'(\alpha)\dd \alpha&\\\notag&-\int_0^{a_i}\int_0^{a^\dagger}\beta^-(\zeta, \alpha)\ell_{0,j}'(\alpha)\dd \alpha\dd \zeta&\\\notag&+\int_0^{a^\dagger}b^-(\alpha)\ell_{0,j}'(\alpha)\dd \alpha,&
i,j=1,\dots N.    
\end{align}

 When the integrals in \eqref{B_N_matrix} and \eqref{M_N_matrix} can not be computed analytically, we approximate them via a quadrature formula. 
For the integrals in $[0, a^\dagger]$, we use the Fejer's first rule quadrature formula \cite{ChebZeros}, and for the integrals in $[0, a_i]$, $i=1,\dots, N$, we use the quadrature weights given by the $i$-th row of the inverse of the differentiation matrix \cite{DiekmannScarabelVermiglio2020}
\begin{equation*}
    (\mathcal D_N)_{ij}:=\ell_{0,j}'(a_i),\quad i,j=1,\dots, N.
    \end{equation*}
 In this case, the bounds in \Cref{corollbound} are preserved under the following regularity conditions (see \cite[pp. 85]{davis2007} and \cite{Trefethen2008}).
\begin{assumption}
\hspace{2mm}
\begin{enumerate}
\item $b^+, b^-\in W^{s, \infty}([0,a^\dagger],\R^{d\times d})$  for some integer $s\ge 1$,
\item $b^+, b^-\in C^\infty([0,a^\dagger],\R^{d\times d})$,
\item $b^+, b^-$ are real analytic.
\end{enumerate} 
\end{assumption} 

In the presence of discontinuities in the model coefficients or in their derivatives, a piecewise approach can be used. In this case, one may choose as discretization points either the Chebyshev zeros extended with the left endpoint, or, to simplify the implementation, the Chebyshev extremal nodes \cite{Mastroianni, Trefethen2008, Trefethen2013}. In the latter case, to approximate the integrals in $[0, a^\dagger]$ we use the Clenshaw--Curtis quadrature formula \cite{ClenshawCurtis1960}. 
Note that the MATLAB demos available at \url{http://cdlab.uniud.it/software} used to make all the tests in \cref{results} implement these piecewise alternatives.

\section{Numerical results}\label{results}
In this section, we apply the numerical method to some instances of \eqref{prototype-model} to experimentally validate the theoretical results in
\cref{convanal}. In particular, we aim to illustrate the link between the convergence order of the numerical approximation $R_N$ of $R$ and the approximation error on the relevant eigenspaces (see \Cref{teospec} and \Cref{remOrder}). 
The latter depends on the regularity of the (generalized) eigenfunctions through the smoothness of the model parameters, see \Cref{corollbound} and \Cref{remOrder}. Thus, the first two examples are constructed to have explicit expressions for $R$ and the associated eigenfunctions. Both examples are scalar and represent the linearized equation for the infected individuals around the disease-free equilibrium. The first example is structured by demographic age (without vertical transmission), whereas the second one is structured by infection age, where the infection process is described by a boundary condition, to analyze the effect of the quadrature error. 
As final example we propose a three-dimensional model structured by demographic age, with horizontal and vertical transmission, and with piecewise $C^\infty$ coefficients estimated from real data.  In this case, analytic expressions for $R$ and the relevant eigenfunctions are not available. Hence, the errors are computed with respect to reference values (obtained with $N=120$). In all these examples, we compare the approximation errors on $R$ by using either the $N$ Chebyshev zeros extended with the left endpoint or the $N+1$  Chebyshev extrema.

\paragraph{Example 1}\label{Example 1}
Motivated by the SIR model structured by demographic age without vertical transmission (see \cite[Chapter II]{Iannelli1995} and \cite[Chapter 6]{Inaba2017}),
we consider the model \eqref{prototype-model} with $d=1$, and parameters
\begin{gather*}
\beta(a, \alpha)=\frac{1}{c }q(a)(a^\dagger-\alpha),\qquad b(a)\equiv 0,\qquad \delta(a)\equiv -\gamma,
\end{gather*}
where $a^\dagger, \gamma>0$,  $q$ is a given function and, for any choice of $q$,
\begin{equation*}
c:=\int_0^{a^\dagger}(a^\dagger-a)\int_0^ae^{-\gamma(a-\alpha)}q(\alpha)\dd \alpha \dd a.
\end{equation*}
With this choice of the parameters, the basic reproduction number is exactly $R_0=1$, and is obtained in our framework by taking $\beta^+=\beta$, $\beta^-\equiv 0$ and $b^+=b^-\equiv 0$.
$\mathcal H$ is a rank-one operator and the eigenfunction relevant to $R=R_0$ is explicitly known: $\psi(a)=\int_0^a q(s)\dd s$, $a\in[0, a^\dagger]$. Moreover, $\varphi(a)=\mathcal M^{-1}\psi(a)=\int_0^a e^{-\gamma(a-\alpha)}\psi(\alpha)\dd \alpha$, $a\in[0, a^\dagger]$.

\Cref{r0Ex1} and \Cref{r0Exnonsmooth} show, for increasing $N$, the error $|R-R_N|$ and the error $|(\mathcal H-\widehat{\mathcal H}_N)\psi_{|_{a=0}}|_{\C^d}+\|((\mathcal H-\widehat{\mathcal H}_N)\psi)'\|_\infty$, which gives a bound on $\varepsilon_N$ in \eqref{epsilon},  with $a^\dagger=\gamma=1$, and for three different choices of $q$, namely $q(a)=e^{-2a}$ (analytic, \Cref{r0Ex1}, left), $q(a)=e^{-(x-0.5)^{-2}}(x-0.5)^{-2}\chi_{[0.5, a^\dagger]}(a)$ ($C^\infty$, \Cref{r0Ex1},  right) and $q(a)=(0.5-a)^2|0.5-a|$ ($W^{3, \infty}$, \Cref{r0Exnonsmooth}).  
The infinite norm is estimated by computing the maximum absolute value over a mesh of $10^4$ equidistant points in $[0,a^\dagger]$. 

In \Cref{r0Ex1}, we observe infinite convergence order, being the relevant eigenfunction $\psi$ either analytic or $C^\infty$, confirming the validity of \Cref{corollbound} under \Cref{ass2} \ref{ass22} and \ref{ass23}, respectively.
\Cref{r0Exnonsmooth} shows order $4$ for the approximation error on $R$ (left and right), order $3$  for $|(\mathcal H-\widehat{\mathcal H}_N)\psi_{|_{a=0}}|_{\C^d}+\|((\mathcal H-\widehat{\mathcal H}_N)\psi)'\|_\infty$ (left), and order $4$ for the error $\|((\mathcal M-\widehat{\mathcal M}^{-1}_N)\mathcal B\varphi)'\|_\infty$ (right), in accordance with \Cref{remOrder}.  
Note that, in all the three cases, the behavior of the approximation error for the Chebyshev extrema is similar to that of the Chebyshev zeros extended with the left endpoint. 

\Cref{r0Ex2} illustrates how the behavior of the approximation error on $R$ depends on the magnitude of the recovery rate $\gamma$ and on the length of the age interval $a^\dagger$.
In particular, we observe that, in the case of large values of $\gamma$ or $a^\dagger$, a larger number of points $N$ is required to obtain small approximation errors. This can be explained from the fact that the approximation error is related to the interpolation error on the exponential function $e^{-\gamma a}$ in the operator $\mathcal{M}^{-1}$, which depends on the derivative of the function (whose norm increases with $\gamma$) and on the length of the interval. For large age intervals, the approximation error can be reduced by resorting to a piecewise approach, splitting the age-interval in smaller regions. 
As an example, in \Cref{piecewiseEx01} we illustrate the results obtained with the piecewise version of the method for the case $a^\dagger=30$ and $\gamma=100$. Therein, we split the age interval in $6$ sub-intervals, and we use a polynomial of degree $N$ in each of them. Note the different behavior between the case of odd and even Chebyshev zeros. 
\begin{remark}\label{remarkgamma}
In real-world applications, both $a^\dagger$ and $\gamma$ may be large. For instance: in a model for human populations structured by demographic age, $a^\dagger$ is typically assumed to be equal or larger than $75$ (yr), while in epidemiology $\gamma$ could be assumed to be larger than $30$ (yr$^{-1}$) for diseases that last less than 10 days on average. See for instance \cite{rubella1985} or \nameref{Example 3}.
\end{remark}

\begin{figure}
\begin{center}
\includegraphics[width=.84\textwidth]{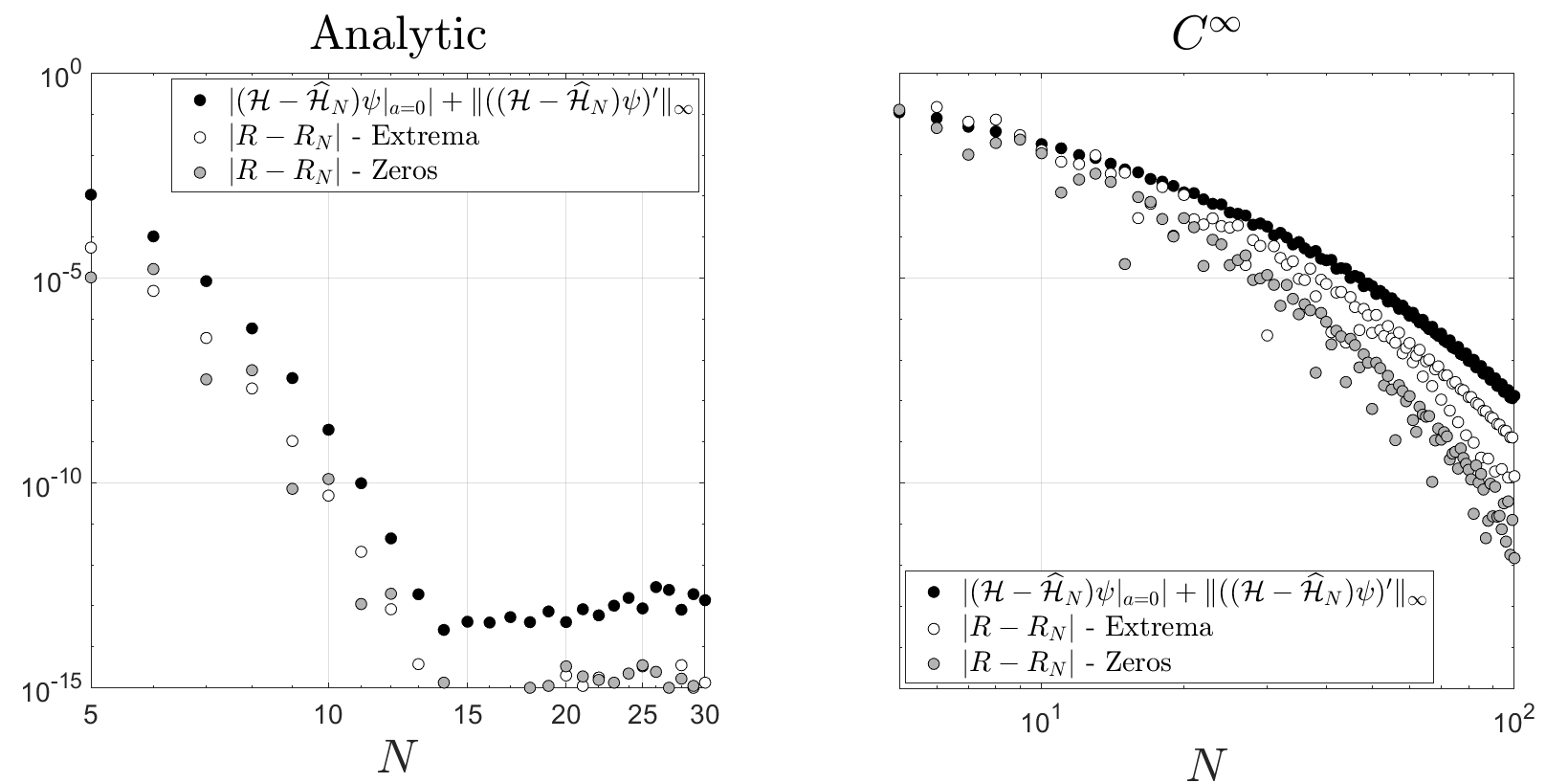}\label{r0Ex1}
\caption{\nameref{Example 1} with $q(a)=e^{-2a}$ (analytic, left) and $q(a)=e^{-(x-0.5)^{-2}}(x-0.5)^{-2}\chi_{[0.5, a^\dagger]}(a)$ ($C^\infty$, right), with $a^\dagger =  \gamma =1$.  Log-log plot for increasing $N$ of the approximation error on $R$ (white dots for the $N+1$ Chebyshev extrema and grey dots for the $N$ Chebyshev zeros extended with the left endpoint) and the  error $|(\mathcal H-\widehat{\mathcal H}_N)\psi_{|_{a=0}}|_{\C^d}+\|((\mathcal H-\widehat{\mathcal H}_N)\psi)'\|_\infty$ for the discretization at the Chebyshev zeros extended with the left endpoint (black dots). Infinite order of convergence is observed in both panels, in agreement with \Cref{corollbound}.}
\includegraphics[width=.84\textwidth]{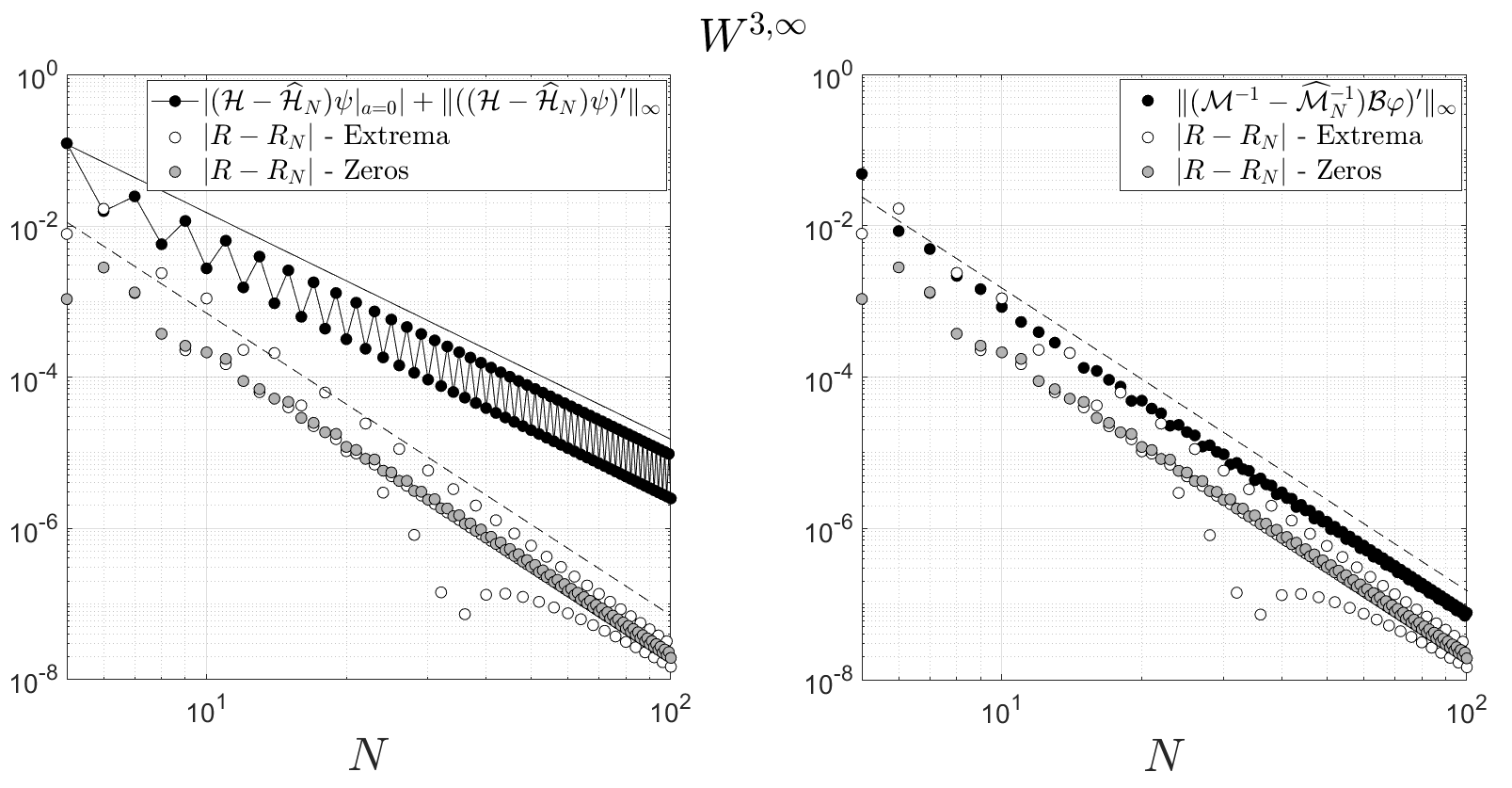}\label{r0Exnonsmooth}
\caption{\nameref{Example 1} with $q(a)=(0.5-a)^2|0.5-a|$ ($W^{3, \infty}$) and $a^\dagger =  \gamma =1$.  Left panel: log-log plot for increasing $N$ of the approximation error on $R$ (white dots for the $N+1$ Chebyshev extrema and grey dots for the $N$ Chebyshev zeros extended with the left endpoint) computed as $\rho(\mathcal H_N)$, and the error $|(\mathcal H-\widehat{\mathcal H}_N)\psi_{|_{a=0}}|_{\C^d}+\|((\mathcal H-\widehat{\mathcal H}_N)\psi)'\|_\infty$ for the discretization at the Chebyshev zeros extended with the left endpoint (black dots). Convergence of order $4$ and order $3$ is observed for $|R-R_N|$ and $|(\mathcal H-\widehat{\mathcal H}_N)\psi_{|_{a=0}}|_{\C^d}+\|((\mathcal H-\widehat{\mathcal H}_N)\psi)'\|_\infty$, respectively. 
Right panel: log-log plot for increasing $N$ of the approximation error on $R$ (white dots for the $N+1$ Chebyshev extrema and grey dots for the $N$ Chebyshev zeros extended with the left endpoint) computed as $\rho(\mathcal M^{-1}_N\mathcal B_N)$,  and the  error $\|((\mathcal M-\widehat{\mathcal M}^{-1}_N)\mathcal B\phi)'\|_\infty$ for the discretization at the Chebyshev zeros extended with the left endpoint. 
Convergence of order $4$ is observed for all the errors, in agreement with \Cref{remOrder}. The dashed lines have slope $-4$, while the solid line has slope $-3$. }
\end{center}
\end{figure}
\begin{figure}
\begin{center}
\includegraphics[width=.9\textwidth]{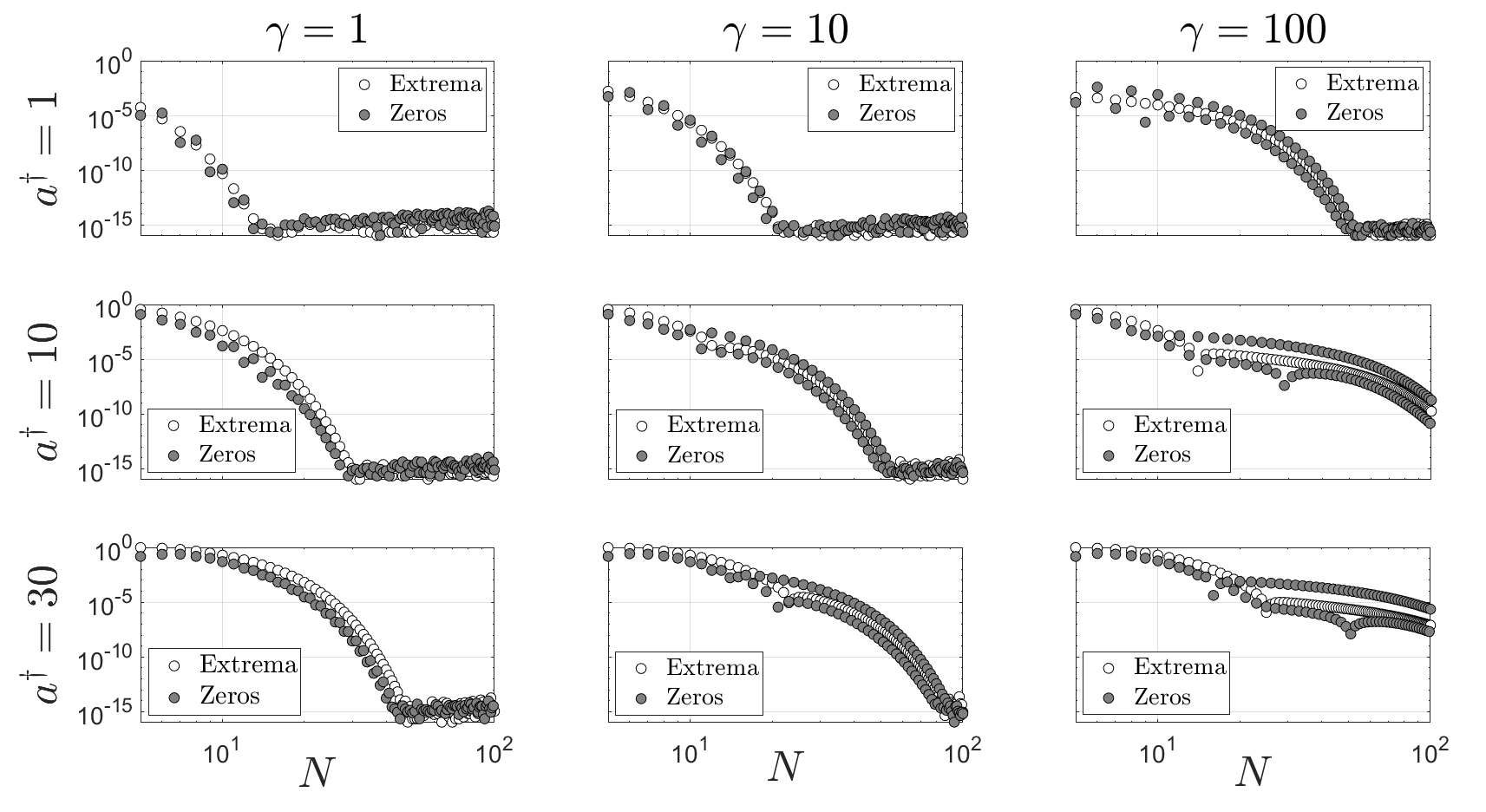}\label{r0Ex2}
\caption{\nameref{Example 1} with  $q(a)=e^{-2a}$. Log-log plot for increasing $N$ of the approximation error on $R$ (white dots for the $N+1$ Chebyshev extrema and grey dots for the $N$ Chebyshev zeros extended with the left endpoint) varying $a^\dagger$ and $\gamma$.}
\end{center}
\end{figure}
\begin{figure}
\begin{center}
\includegraphics[width=.43\textwidth]{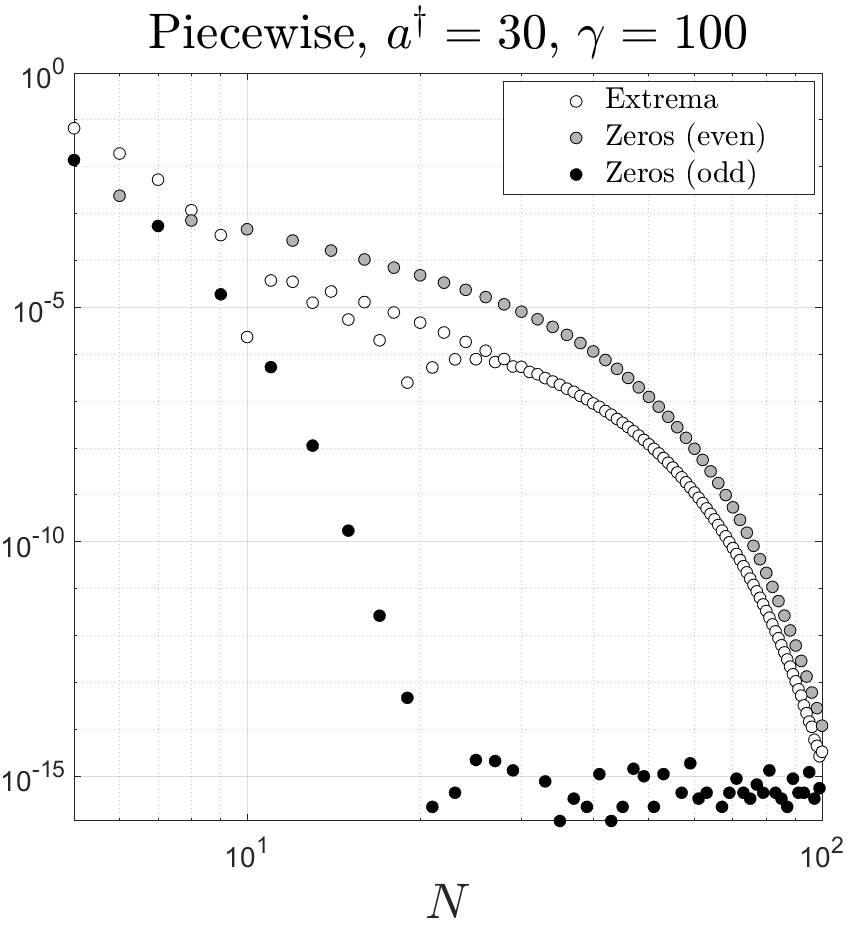}\label{piecewiseEx01}
\caption{\nameref{Example 1} with $q(a)=e^{-2a}$, $a^\dagger =30$ and $\gamma =100$. Log-log plot for increasing $N$ of the piecewise approximation error on $R$ (white dots for the $N+1$ Chebyshev extrema,  grey and black dots for the $N$ Chebyshev zeros extended with the left endpoint with $N$ even and  $N$ odd, respectively). Note the different behavior of the error for odd and  even $N$ for the Chebyshev zeros.}
\end{center}
\end{figure}

\paragraph{Example 2}\label{Example 2}
We consider an SIR model where infected individuals are structured by infection age, see for example \cite[Chapter 7]{Iannelli1995}.
The linearized equation for the infected individuals around the disease-free equilibrium
can be recast in \eqref{prototype-model} by taking $d=1$, $\beta(a, \alpha)\equiv 0$ and $\delta(a)\equiv -\gamma$, for $\gamma>0$.
In the following, we do not consider the presence of control measures. 

To investigate the convergence of our method, we consider the following explicit expression for the basic reproduction number $R_0$ \cite[section 5.3]{Inaba2017}:
\begin{equation}\label{R0TSI}
R_0=\int_0^{a^\dagger}\beta(a)e^{-\int_0^a \gamma(\alpha)\dd \alpha}\dd a.
\end{equation}
We take $a^\dagger=14$, and  $b$ and $\gamma$ such that $b(a)e^{-\gamma a} = \Gamma(k, \theta)(a)$, where $\Gamma(k, \theta)$ is a truncated Gamma density function with shape parameter $k>0$ and scale parameter $\theta>0$, normalized in the interval $[0, a^\dagger]$ \cite[section 4.1]{de2024approximating}. In particular we take $b(a)=ca^{k}$, $\gamma =1/\theta$, and $\theta=0.25$,  where $c:=\|\Gamma(k, \theta)\|_{X}^{-1}$.
For this choice, from \eqref{R0TSI}, we get that $R_0=1$. 
In our framework, we take $b^+=b$ and $b^-\equiv0$. The relevant eigenfunction is constant.

\Cref{tsifig} shows the behavior of the approximation error on $R=R_0$ for $k=2$ and $k=\pi$, see \cite[Table 1]{scarabel2023}. The method converges with infinite order for $k=2$, i.e., when $b$ is of class $C^\infty$. Finite convergence order is observed for $k=\pi$, i.e., when $b''$ has a pole in $a=0$. This illustrates how the quadrature errors affect the convergence order on $R$, even in the case of a constant eigenfunction. As pointed out in \cite{scarabel2023}, in this case the order of convergence can be improved by computing the integrals, for example, with the  MATLAB built-in  integral function. 
\begin{figure}
\begin{center}
\includegraphics[width=.84\textwidth]{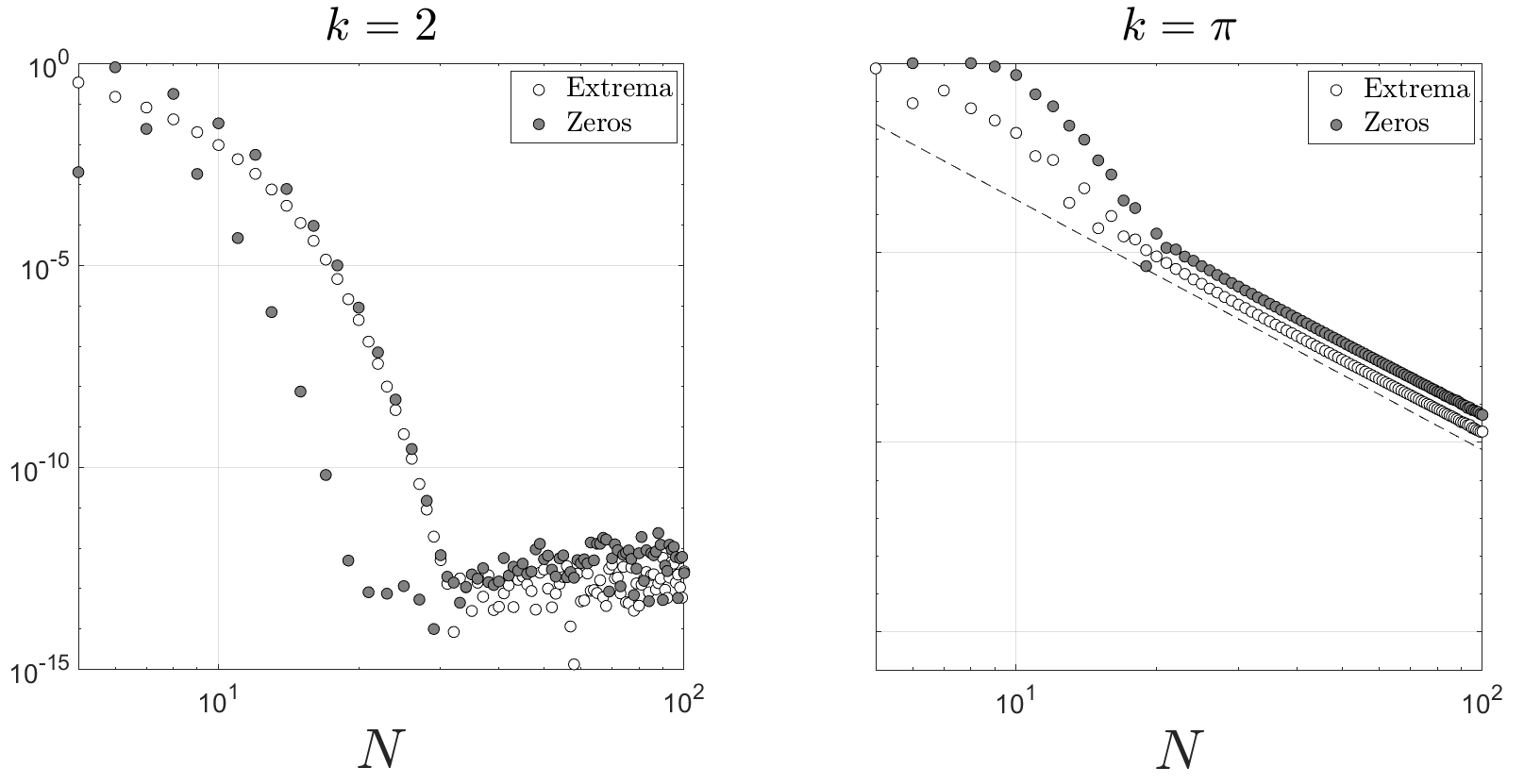}\label{tsifig}
\caption{\nameref{Example 2}. Log-log plot of the approximation error for increasing $N$ on $R=1$ (white dots for the $N+1$ Chebyshev extrema and grey dots for the $N$ Chebyshev zeros extended with the left endpoint), for $k=2$ (left) and $k=\pi$ (right). The slope of the dashed line is $-6.5$.}
\end{center}
\end{figure}

\paragraph{Example 3}\label{Example 3}
We consider a model inspired by \cite{Zhao2000, Zou2010} for the spread of Hepatitis B (HBV) in China, and we refer to \cite{reviewHBV} for a recent review of other models of the literature. 
Let $S(t,a)$, $L(t, a)$, $I(t,a)$, $C(t,a)$, $R(t,a)$ and $V(t,a)$ denote the density of individuals that are susceptible, latent (infected but not infectious), acutely infected, chronically infected, recovered, and vaccinated, respectively, at time $t\ge 0$ and demographic age $a\in [0, a^\dagger]$. The model reads
\begin{equation*}
\left\{\setlength\arraycolsep{0.1em}\begin{array}{rlll} 
\mathcal D S(t, a)&=&\omega V(t, a)-(\mu(a)+\lambda(t, a)+\nu(a))S(t, a), &\\[3mm]
\mathcal D L(t, a)&=&\lambda(t,a) S(t, a)-(\mu(a)+\sigma)L(t, a), &\\[3mm]
\mathcal D I(t, a)&=&\sigma L(t, a)-(\mu(a)+\gamma_1 )I(t, a), &\\[3mm]
\mathcal D C(t, a)&=&p(a)\gamma_1 I(t, a)-(\mu(a)+\gamma_2(a))C(t, a)
, &\\[3mm]
\mathcal D R(t, a)&=&\gamma_2(a)C(t, a)+(1-p(a))\gamma_1 I(t, a)-\mu(a)R(t, a)
, &\\[3mm]
\mathcal D V(t, a)&=&\nu(a)S(t, a)-(\omega+\mu(a))V(t, a), &\\
\end{array} 
\right.
\end{equation*}
for $t\ge 0$ and $a\in [0, a^\dagger]$, with boundary conditions \cite{Edmunds1996}
\begin{equation*}
\left\{\setlength\arraycolsep{0.1em}\begin{array}{rlll} 
S(t, 0)&=& \theta\displaystyle\int_0^{a^\dagger} f(a)\left(P(t,a)-q_1 I(t,a) -q_2 C(t,a)\right)\dd a, &\quad \\
L(t, 0)&=&\theta\displaystyle\int_0^{a^\dagger} f(a)\left(q_1 I(t,a) +q_2 C(t,a)\right)\dd a, &\quad \\
V(t, 0)&=&(1-\theta)\displaystyle\int_0^{a^\dagger} f(a) P(t,a)\dd a, &\quad \\[3.5mm]
I(t, 0)&=&C(t, 0)=R(t,0)=0, &\quad 
\end{array}\right.
\end{equation*}
for $t\ge 0$, where \cite{Edmunds1996, AndersonHBV}
\begin{equation*}
\lambda(t,a)=\int_0^{a^\dagger} \hat k(a, \alpha)\cfrac{I(t, \alpha)+\epsilon C(t, \alpha)}{\int_0^{a^\dagger}P(t, a)\dd a}\dd \alpha,    
\end{equation*}
and $P:=S+L+I+R+V+C.$
Here, $\mu$ is the natural  mortality rate (yr$^{-1})$, 
$f$ is the per capita birth rate (yr$^{-1})$, $\hat k$ is the per capita transmission rate (yr$^{-1})$, $\epsilon$ is the relative transmission rate of chronic carriers, $\sigma$ is the rate of moving from latent to infectious phase (yr$^{-1})$, $\gamma_1$ is the rate of moving from acute infection to recovered or chronic (yr$^{-1})$, $p$ is the probability of becoming chronic instead of recovering, $\gamma_2$ is the recovery rate of chronic infection (yr$^{-1})$,
$\nu$ is the per capita vaccination rate of individuals of age $a>0$ (yr$^{-1})$, $\omega$ is the rate of waning of vaccine-induced immunity (yr$^{-1})$, $q_1$ and $q_2$ are the fraction of perinatally infected from individuals in the acute and chronic phase, respectively, and $\theta$ is the fraction of failed vaccinations at birth.
We assume  that the host population is at a demographic steady state, i.e., $\int_0^{a^\dagger}f(a)\Pi(a)\dd a=1$ holds for $\Pi(a):=\exp(-\int_0^a \mu(\xi)\dd \xi)$, and that it has already attained the stable age distribution $$P(t,a)\equiv P^*(a)=P_0\Pi(a)\left(\int_0^{a^\dagger}\Pi(\xi)\dd\xi\right)^{-1}, \qquad a \in [0, a^\dagger],$$ for some $P_0>0$ \cite[Chapter 8]{Inaba2017}. Then, by defining
$s:=S/P$, $l:=L/P$, $i:=I/P$, $c:=C/P$, $r:=R/P$,  and $v:=V/P$,
the resulting model has the disease-free equilibrium $(s^*,l^*,i^*, c^*, r^*, v^*)=(s^*,0,0, 0, 0, 1-s^*),$ where 
\begin{equation*}
s^*(a)=\theta{\e}^{-\omega a-\int_0^a \nu(\alpha) \dd\alpha}+\omega\displaystyle\int_0^a {\e}^{-\omega (a-s)-\int_s^a \nu(\alpha) \dd\alpha}\dd s,\qquad a \in [0, a^\dagger]. 
\end{equation*}
Observe that, in the absence of vaccination ($\nu\equiv 0$ and $\theta=1$), we have $s^*\equiv 1$.
The linearized equations for the infected individuals around the disease-free equilibrium read
\begin{equation}\label{HBVlin}
\left\{\setlength\arraycolsep{0.1em}\begin{array}{rlll} 
\mathcal D l(t, a)&=& s^*(a) \displaystyle\int_0^{a^\dagger} k(a, \alpha)(i(t, \alpha)+\epsilon c(t, \alpha))\dd \alpha -\sigma l(t, a)
, &\\[3.5mm]
\mathcal D i(t, a)&=&\sigma l(t, a)-\gamma_1 i(t, a),&\\[3.5mm]
\mathcal D c(t, a)&=& p(a)\gamma_1 i(t, a)-\gamma_2(a) c(t, a),&\\[0.5mm]
l(t, 0)&=&\theta\displaystyle\int_{0}^{a^\dagger}  f(a) \Pi(a) \left(q_1 i(t, a)+q_2 c(t,a)\right)\dd a,&\\[3.5mm]
i(t, 0)&=&0,&\\[3.5mm]
c(t, 0)&=& 0,&
\end{array} 
\right.
\end{equation}
for $t\ge 0,\ a\in [0, a^\dagger]$, where $k(a, \alpha):=\hat k(a, \alpha)\Pi(a)(\int_0^{a^\dagger}\Pi(\xi)\dd\xi)$.
\eqref{HBVlin} can be recast in \eqref{prototype-model} by taking 
\begin{align*}
\beta(a, \alpha)=
s^*(a)k(a, \alpha)\begin{pmatrix}
0 & 1 & \epsilon \\
0 & 0 & 0 \\
0 & 0 & 0
\end{pmatrix},\quad
b(a)=\theta f(a)\Pi(a)
\begin{pmatrix}
0 & q_1 & q_2 \\
0 & 0 & 0 \\
0 & 0 & 0
\end{pmatrix},
\end{align*}
and
\begin{equation*}
\delta(a)=\begin{pmatrix}
-\sigma & 0 & 0 \\
\sigma & -\gamma_1 & 0 \\
0 & p(a)\gamma_1 & -\gamma_2
\end{pmatrix}.
\end{equation*}
For this model, we can compute three reproduction numbers: the basic reproduction number $R_0$, for $\beta^+=\beta$, $\beta^-\equiv 0$, $b^+=b$, $b^-\equiv 0$; the type reproduction number for horizontal transmission $T_H$, for $\beta^+=\beta$, $\beta^-\equiv 0$, $b^+\equiv 0$, $b^-= b$; and the type reproduction number for vertical transmission $T_V$, for $\beta^+\equiv 0$, $\beta^-=\beta$, $b^+=b$, $b^-\equiv 0$.

Following \cite{Edmunds1996, Zhao2000, Zou2010}, we assume $a^\dagger=75$, $\Pi\equiv 1$, $\epsilon = 0.16$, $\sigma=6$, $\gamma_1=4$, $\gamma_2=0.025$, $\omega=0.1$, $q_1=0.711$, $q_2=0.109$,
\begin{equation*}
p(a)=0.176501\exp(-0.787711a)+0.02116,\qquad a\in[0, a^\dagger],
\end{equation*}
$f(a)=0.018\chi_{[18, a^\dagger]}(a)$ for $a\in[0, a^\dagger]$, and we vary $\nu, \theta$ in $[0, 1]$. As for $k$, we estimate it from real data. In \cite[Formula 2]{Zhao2000} the authors give the following form of the force of infection 
\begin{equation}\label{foihbvformula}
    \lambda(a):=
    \left\{\setlength\arraycolsep{0.1em}\begin{array}{lll} 
    0.13074116&-1.362531\cdot 10^{-2}a\\
    &+4.6463\cdot 10^{-4}a^2- 4.89\cdot 10^{-6}a^3, &\quad a\in [0, 47.5],\\[3mm]
    \lambda(47.5),& & \quad a\in(47.5, a^\dagger],
\end{array} 
\right.
\end{equation}
which was estimated from serological data by applying the procedure described in \cite{grenfell1985}.
Here, in order to estimate $k$, we assume that it is piecewise constant among different age-groups, i.e.,
\begin{equation*}
k(a, \alpha)=k_{ij},\quad \text{for}\quad (a, \alpha)\in [\bar a_{i-1}, \bar a_i)\times[\bar a_{j-1}, \bar a_j),\quad i,j=1,\dots, 7,
\end{equation*}
where the age-groups are listed in \Cref{hbvFOIPCW}.
This gives us a \emph{Who Acquires Infection From Whom} (WAIFW) matrix $(k_{ij})_{i,j=1,\dots, 7}$, which can be estimated
by applying the well-kwown procedure of \cite[Appendix A]{Anderson1985}.\footnote{We estimate $k$ by using $\gamma_1$ as recovery parameter and by neglecting, in first approximation, the role of chronic carriers in the transmission pattern.} For doing this, we need to assume a particular form for the WAIFW matrix (otherwise the estimation problem is over-determined). Here we chose the one described in \cite[Table II]{Edmunds1996} and we refer to \cite[Appendix A]{Anderson1985} for other possible choices. More in details,  $(k_{ij})_{i,j=1,\dots, 7}$ is assumed to be symmetric with elements $k_{ij}=k_i$ for $i\ge j$.
\begin{table}
\footnotesize\caption{Age-specific forces of infection $\lambda_i$'s (yr$^{-1}$) derived from \cite{Zhao2000}, and corresponding $k_i$'s defining the entries of the WAIFW, for $i=1,\dots,7$.}
\label{hbvFOIPCW}
\begin{center}
\begin{tabular}{c | c  c  c c c c c}
\rowcolor{gray!20}
Age class (years) & $0-2$  & $3-5$ & $6-9$ & $10-14$ & $15-29$ & $30-49$ & $50-75$\\
\hline
$\lambda_i$ (assumed) & $0.112$ & $0.079$ & $0.049$ & $0.024$ & $0.006$ & $0.013$ & $0.008$ \\
$k_{i}$ (computed) & $1.070$ & $0.607$ & $0.338$ & $0.149$ & $0.027$ & $0.068$ & $0.041$ \\
\hline
\end{tabular}
\end{center}
\end{table}
Then, in order to simplify the estimation of $k$, we take the mean values among different age-intervals of the age-specific force of infection in \eqref{foihbvformula}, see \Cref{hbvFOIPCW}.
The age groups are chosen by merging the original age-group division considered in \cite{Zhao2000}, so that the piecewise force of infection captures the main geometrical features of \eqref{foihbvformula}.
\begin{figure}
\begin{center}
\includegraphics[width=1.\textwidth]{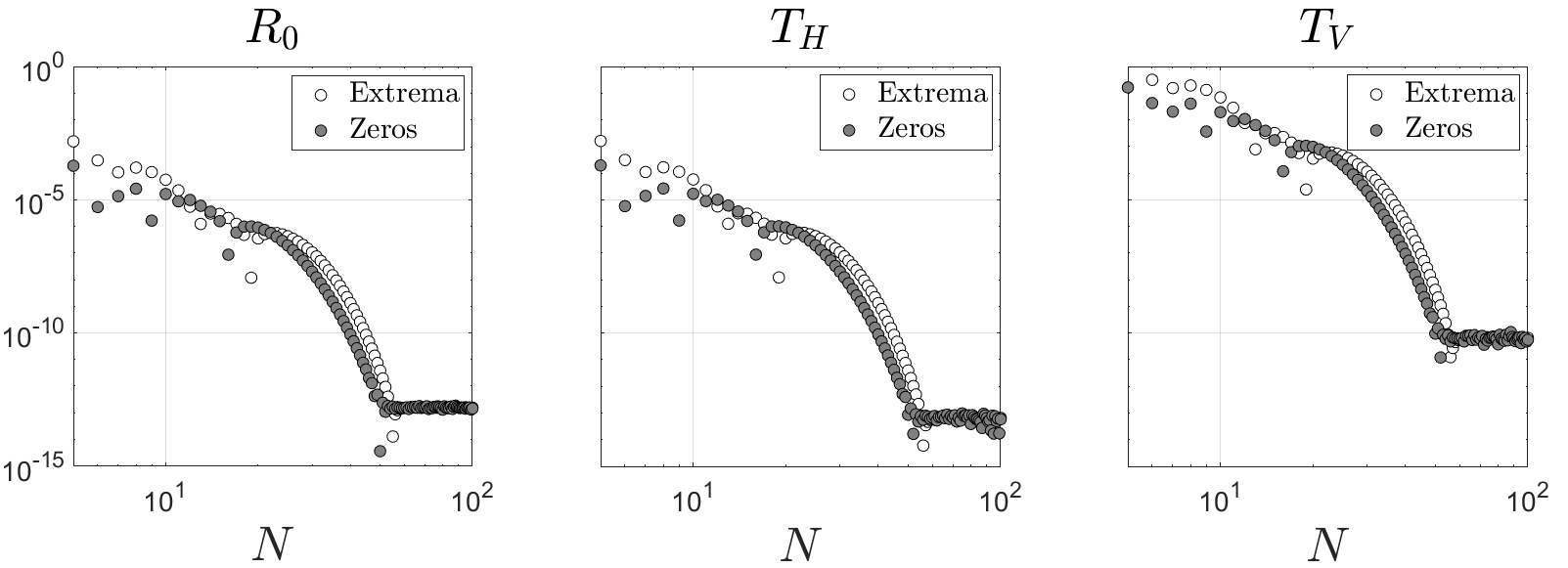}\label{hbvconv}
\caption{\nameref{Example 3}. Log-log plot of the absolute approximation errors for increasing $N$ on $R_0$ (left), $T_H$ (center) and $T_V$ (right), with $\nu=0.1$ and $\theta=0.59$ (white dots for the $N+1$ Chebyshev extrema and grey dots for the $N$ Chebyshev zeros extended with the left endpoint). The reference values $R_0\approx 1.048182936983250$, $T_H\approx 1.004493064088357$, and $T_V\approx 2.765546573797665$, are obtained with $N=120$.}
\end{center}
\end{figure}
\begin{figure}
\begin{center}
\includegraphics[width=.84\textwidth]{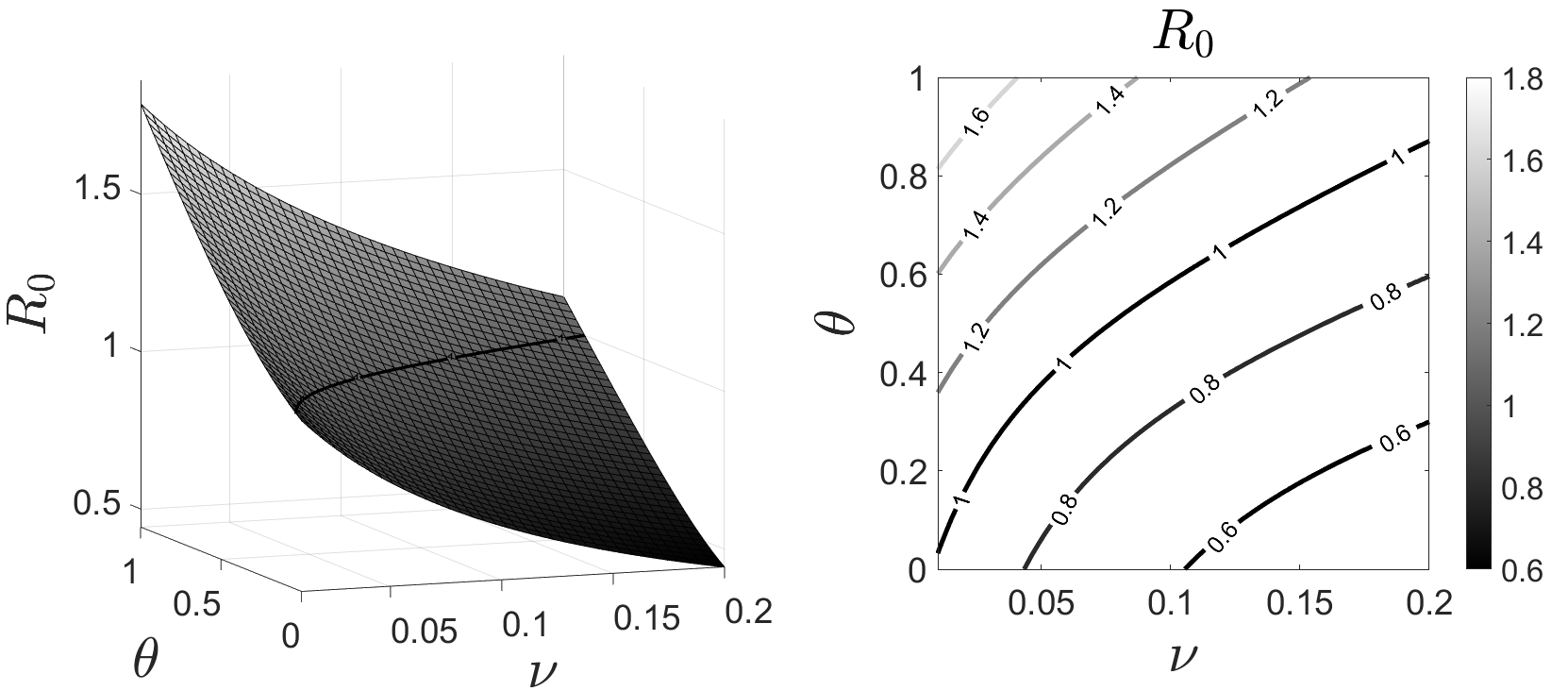}\label{HBVR0}
\caption{\nameref{Example 3}. $R_0$ as a function of $\nu$ and $\theta$ (left) and relevant level curves (right) computed with $N=50$.}
\end{center}
\end{figure}

\Cref{hbvconv} shows that the approximating reproduction numbers converge with infinite order, although $N \gtrsim 20$
nodes in each sub-interval are required to appreciate this behavior. This can be explained by the fact that, even though we are using a piecewise approach here, some of the age sub-intervals are large, as already discussed in \nameref{remarkgamma}.

\Cref{HBVR0} shows a practical application of the method. Therein, the behavior of $R_0$ as a function of the fraction of failed vaccinations at birth $\theta$ and the per-capita vaccination rate $\nu$ is investigated. This shows that, even in the presence of vaccination, a large fraction of failed vaccinations at birth could lead to the spread of the epidemic, while this could be prevented for larger values of $\nu$. Let us note that this result extends the one presented in \cite[Figure 3]{Zou2010}, where the behavior of $R_0$ is investigated by neglecting the effect of vertical transmission.

\section*{Acknowledgments}
The authors are members of the INdAM Research group GNCS and of the UMI Research group ``Mo\-del\-li\-sti\-ca so\-cio-epi\-de\-mio\-lo\-gi\-ca''.
The work of Simone De Reggi and Rossana Vermiglio was supported by the Italian Ministry of University and Research (MUR) through the PRIN 2020 project (No. 2020JLWP23) ``Integrated Mathematical Approaches to Socio-Epidemiological Dynamics'', Unit of Udine (CUP: G25F22000430006).

 \printbibliography

\end{document}